\documentclass[12pt]{amsart}

\usepackage{latexsym, amssymb, amsmath,esint}

\usepackage{amsfonts, graphicx}
\usepackage{graphicx,color}
\newcommand{\be}{\begin{equation}}
\newcommand{\ee}{\end{equation}}
\newcommand{\beq}{\begin{eqnarray}}
\newcommand{\eeq}{\end{eqnarray}}

\newtheorem{thm}{Theorem}[section]

\newtheorem{lma}{Lemma}[section]
\newtheorem{prop}{Proposition}[section]
\newtheorem{cor}{Corollary}[section]

\theoremstyle{remark}
\newtheorem{rem}{Remark}[section]
\numberwithin{equation}{section}

\def\be{\begin{equation}}
\def\ee{\end{equation}}
\def\bee{\begin{equation*}}
\def\eee{\end{equation*}}

\def\lf{\left}
\def\ri{\right}

\def\red{\color{red}}
\def\green{\color{green}}
\def\blue{\color{blue}}

\def\Ric{\text{\rm Ric}}
\def\Rm{\text{\rm Rm}}

\def\la{\langle}
\def\ra{\rangle}
\def\p{\partial}

\def\heat{\lf(\frac{\p}{\p t}-\Delta\ri)}

\def\e{\varepsilon}
\def\a{{\alpha}}
\def\b{{\beta}}

\def\R{\mathbb{R}}

\begin{document}

\title[]
{Small curvature concentration and Ricci flow smoothing}

 \author{Pak-Yeung Chan}
\address[Pak-Yeung Chan]{Department of Mathematics,
University of California, San Diego,
La Jolla, CA 92093}
\email{pachan@ucsd.edu}

 \author{Eric Chen}
\address[Eric Chen]{
Department of Mathematics,
University of California,
Santa Barbara CA 93106-3080, USA}
\curraddr{Department of Mathematics, University of California, Berkeley CA 94720-3840, USA}
\email{ecc@berkeley.edu}

 \author{Man-Chun Lee}
\address[Man-Chun Lee]{Department of Mathematics, Northwestern University, 2033 Sheridan Road, Evanston, IL 60208; Mathematics Institute, Zeeman Building,
University of Warwick, Coventry CV4 7AL}
\curraddr{Department of Mathematics, The Chinese University of Hong
Kong, Shatin, Hong Kong, China}
\email{mclee@math.cuhk.edu.hk}

\renewcommand{\subjclassname}{
  \textup{2010} Mathematics Subject Classification}

\date{\today}

\begin{abstract}
We show that a complete Ricci flow of bounded curvature which begins from a manifold with a Ricci lower bound, local entropy bound, and small local scale-invariant integral curvature control will have global point-wise curvature control at positive times which only depends on the initial almost Euclidean structure. As applications, we use the Ricci flows to study the diffeomorphism type of manifolds and the regularity of Gromov-Hausdorff limit of manifolds with small curvature concentration. 

\end{abstract}


\maketitle

\markboth{Pak-Yeung Chan, Eric Chen and Man-Chun Lee}{}

\section{Introduction}

The Ricci flow deforms a metric $g$ on a Riemannian manifold $(M^n,g)$ according to the equation
\begin{align*}
    \frac{\partial}{\partial t}g(t)=-2\Ric_{g(t)}.
\end{align*}
Since its introduction by Hamilton \cite{Hamilton-1982}, the Ricci flow has been used in a wide variety of settings to regularize metrics. One sense in which this occurs is described by Perelman's pseudolocality theorem \cite{Perelman-2002}, which has played a crucial role in work on the short time existence of the Ricci flow, especially in settings where the initial data lacks bounded curvature or completeness \cite{Topping2010, He2016}.
\begin{thm} \cite[Theorem 10.1]{Perelman-2002}\label{Perelman} For any $\a>0$, there exist  positive constant $\e$ and $\delta$ such that if ($M^n, g(t))$, $t$ $\in [0,\e r_0]$  is a solution to the Ricci flow for some $r_0>0$ and in addition 
\begin{itemize}
\item $R\geq -r_0^{-2}$ on $B_0(x_0, r_0)$;
\item $|\partial \Omega|^n \geq (1-\delta) c_n |\Omega|^{n-1}$, for any open set in $B_0(x_0, r_0)$, where $c_n$ is the isoperimetric constant of $\R^n$,
\end{itemize}
then for any $t$ $\in [0, (\e r_0)^2]$ and $x$ $\in B_t(x_0, \e r_0)$
\begin{equation*}
|\Rm|(x, t)\leq \a t^{-1}+ (\e r_0)^{-2}.
\end{equation*}
\end{thm}
Thus, Perelman's pseudolocality tells us that given a lower Ricci bound on an almost Euclidean region, we can deduce regularization in the sense of curvature control along the Ricci flow for short times. Since Perelman's work, many extensions have been developed in a variety of settings \cite{Chen-2009,ChauTamYu2012,TianWang2015,Wang-2020}.


Related regularization results for the Ricci flow under critical $L^{n/2}$ bounds of $\Rm$ have previously been studied in \cite{Yang-2011, Li-2012} assuming also pointwise two-sided bounds on $|\Ric|$, and in \cite{Wang-2011} assuming alternatively a supercritical $\|\Ric\|_p$, $p>n/2$ bound.  In this note, we will consider conditions which are scaling invariant. In particular, we study the flow under local critical $L^{n/2}$ bounds of $\Rm$, but will instead do so in combination with a Ricci lower bound and control of the local entropy, a localization of Perelman's entropy introduced by Wang \cite{Wang-2018}. 

Below we state our main result, referring to the beginning of Section \ref{CurvEst} for most of the associated notation. Throughout, we will use $a\wedge b$ to denote $\min\{a,b\}$.

\begin{thm}\label{pseudo}
For all $A,\lambda>0$, there are $C_0(n,A,\lambda)$, $\sigma(n,A)$ and $\hat T(n,A,\lambda)>0$ such that the following holds. Suppose $(M^n,g(t))$ is a complete Ricci flow of bounded curvature on $[0,T]$ and the initial metric $g(0)$ satisfies the followings:
\begin{enumerate}
\item[(a)] $\Ric(g(0))\geq -\lambda;$
\item[(b)] $\nu(B_{g(0)}(p,5),g(0),1)\geq -A$;
\item[(c)] $\left(\int_{B_{g(0)}(p, 2)}|\Rm(g_0)|^{n/2}d\mu_0\right)^{2/n}\leq \e $ for some $\e<\sigma $,
\end{enumerate}
for all $p\in M$. Then we have for any $x$ $\in M$ and  $t$ $\in (0,  T\wedge \hat T]$,
\be\label{C0 bdd for Riem}
|\Rm|(x,t)\leq\frac{C_0\e}{t} \quad\text{and}\quad \mathrm{inj}_{g(t)}(x) \geq C_0^{-1}\sqrt{t}.
\ee
Moreover, we have $\left(\int_{B_t(x,1)}|\Rm(g(t))|^{n /2}d\mu_t\right)^{2/n}\leq C_0 \e$ for all $x\in M$. In particular, the Ricci flow must exist up to $\hat T$.
\end{thm}

\begin{rem}
In the statement above, we choose the scale $1$ in the local entropy condition only for convenience. 
\end{rem}

Theorem \ref{pseudo} is a smoothing result also based on an initial ``almost Euclidean'' assumption. However, instead of characterizing this using the isoperimetric constant as in Theorem \ref{Perelman}, we instead use rough non-collapsing and small curvature concentration.

Using ideas related to those used to prove Theorem \ref{pseudo} and point-picking technique, we prove a gap result for steady and shrinking gradient Ricci solitons without assuming curvature boundedness.  

\begin{thm}\label{gap thm for soliton}
For all $A>0$, there is $\e(n,A)>0$ such that if $(M^n,g, f)$ is a complete shrinking or steady gradient Ricci soliton satisfying 
\begin{enumerate}
    \item[(a)] $\nu(M,g)\geq -A$;
    \item[(b)] $\left(\int_M |\Rm|^{n/2} d\mu \right)^{2/n}\leq \e$,
\end{enumerate}
then $(M,g)$ is isometric to the standard Euclidean space $\mathbb{R}^n$.
\end{thm}
{ 
\begin{rem}
The smallness condition (b) in Theorem \ref{gap thm for soliton} is necessary. Indeed, as pointed out by Chen-Zhu \cite{ChenZhu-2002}, the Eguchi-Hanson metric $(M,g)$ is Ricci flat with $|\Rm(x)|\leq C(d_g(x,p)+1)^{-6}$ and has Euclidean volume growth.
Therefore it satisfies $||\Rm||_{L^{n/2}}<\infty$ and $\nu(M,g)>-\infty$ by Remark \ref{volume ratio and Ricci bdd imply entropy lower bdd} and a scaling argument. However, it is not isometric to Euclidean flat metric.
\end{rem}
}

Gap results for gradient Ricci solitons have been previously studied for instance in \cite{Yokota2009,FG2012,Zhang2018} under global assumptions on the potential function $f$, sometimes along with pointwise curvature control (see also \cite{{MunteanuWang-2011},{ChenDeruelle-2015}, {DengZhu-2015}, {Chan-2020}} ).

Theorem \ref{pseudo} lends itself to several applications. First, we have a gap result for Ricci-nonnegative Riemannian manifolds with $\|\Rm\|_{L^{n/2}}$ small.


\begin{cor}\label{Corollary-diff} 
For all $A>0$, there is $\sigma(n,A)>0$ such that if $(M^n,g)$ is a complete noncompact manifold with 
\begin{enumerate}
\item bounded curvature
    \item $\Ric\geq 0$;
    \item\label{cornu} $\nu(M,g)\geq -A$ ;
    \item\label{corcurv} $\left(\int_M |\Rm|^{n/2} d\mu_g \right)^{2/n}\leq \sigma$.
\end{enumerate}
Then $g$ is of Euclidean volume growth. Moreover, $M^n$ is diffeomorphic to $\mathbb{R}^n$.
\end{cor}

There is a large body of work on Ricci-nonnegative noncompact manifolds, and several results show that under some additional assumptions (such as almost Euclidean volume growth), such manifolds must be diffeomorphic to $\mathbb{R}^n$ \cite{BKN1989,CheegerColding1997,Xia1999}. Corollary \ref{Corollary-diff} is related to two results of this kind by Ledoux and Xia \cite{Ledoux1999,Xia2001}, which assert that a complete, Ricci-nonnegative manifold with Euclidean-type Sobolev constant close to that of Euclidean space must be diffeomorphic to $\mathbb{R}^n$. Indeed, Condition (\ref{cornu}) of Corollary \ref{Corollary-diff} above can be seen as a weakening of this requirement, since it holds as long as there is some constant which makes the Euclidean-type Sobolev inequality valid. This is compensated for by Condition (\ref{corcurv}) on the smallness of the total scale-invariant curvature.

Carron has also pointed out to us that Corollary \ref{Corollary-diff} is related to the following statement which can be derived from works of Cheeger \cite{Cheeger2003} and Cheeger-Colding \cite{CheegerColding1997}: If a complete noncompact manifold with $\Ric\geq 0$ has $\int_M|\Rm|^{n/2}d\mu_g$ sufficiently small relative its asymptotic volume ratio (assumed nonzero), then it must be diffeomorphic to $\mathbb{R}^n$. Indeed, these hypotheses ensure by \cite[Theorem 4.32]{Cheeger2003} that the manifold must have asymptotic volume ratio close to one, from which one can conclude that the manifold must be diffeomorphic to $\mathbb{R}^n$ by \cite[Theorem A.1.11.]{CheegerColding1997} (see also \cite[Theorem 5.7]{Wang-2020} for a recent proof via Ricci flow). Our assumptions in Corollary \ref{Corollary-diff} differ slightly from this statement's because our required smallness of $\sigma$ is relative to a lower bound on entropy. Although we will prove below that under bounded curvature, bounded entropy, and small $||\Rm||_{L^{n/2}}$ we indeed have almost Euclidean volume growth, it is unclear to us whether an entropy lower bound implies an asymptotic volume ratio lower bound in general. The method here relies on the existence of Ricci flow with uniform estimates. It will be interesting to know if the curvature boundedness is necessary to ensure its existence in general.




We can also apply Theorem \ref{pseudo} to obtain a finite diffeomorphism-type result and a  Gromov--Hausdorff compactness result in the setting of length spaces.


\begin{cor}\label{diffeotype} For all $A>0$, there is $\sigma(n,A)>0$ such that for $C_1,C_2$, the space of compact manifolds $(M,g)$ satisfying 
\begin{enumerate}
    \item[(a)] $\Ric(g)\geq -C_1$;    
    \item[(b)] $Vol_{g}(M)\leq C_2$;
    \item[(c)] $\inf_{p\in M}\nu(B_g(p,5),g,1)\geq -A$;
    \item[(d)] $\sup_{p\in M}\left(\int_{B_g(p,2)}|\Rm|^{n/2} d\mu_g\right)^{2/n} \leq \sigma $
\end{enumerate}
contains finitely many diffeomorphism types.
\end{cor}

Corollary \ref{diffeotype} follows via an argument analogous to the that in the proof of \cite[Theorem 37.1]{KleinerLott2008}, which was proved by Perelman \cite[Remark 10.5]{Perelman-2002} using Perelman’s pseudolocality. In our case, the use of Perelman’s pseudolocality is replaced by Theorem~\ref{pseudo}.

\begin{thm}\label{Theorem:GH} For any positive integer $n\geq 3$ and constant $A\geq 1000n$, there exists constant $\e_0(n,A)$ such that the following holds. Suppose $( M^n_i,g_i, p_i)$ is a pointed sequence of Riemannian manifolds with the following properties:
\begin{enumerate}
\item[(a)] $(M_i,g_i)$ has bounded curvature; 
\item[(b)] $\Ric(g_i)\geq -\lambda;$
\item[(c)] $\left(\int_{B_{g_i}(q, 2)}|\Rm(g_i)|^{n/2}d\mu_i\right)^{2/n}\leq \e_0$ for all $q$ in $M_i$ ;
\item[(d)] $\nu(B_{g_i}(q, 5), g_i,1)\geq -A$,  for all $q$ in $M_i$.
\end{enumerate}
Then there exists a smooth manifold $M_\infty$ and a complete distance metric $d_\infty$ on $M_\infty$ generating the same topology as $M_\infty$ such that after passing to sub-sequence, $(M_i, d_{g_i}, p_i)$ converges in pointed Gromov Hausdorff sense to $(M_{\infty}, d_{g_{\infty}}, p_{\infty})$.
\end{thm}

\begin{rem}
The Ricci lower bound assumption on the initial metric can in fact be weakened to a scalar curvature lower bound and uniform volume doubling for some fixed scale. But we feel it is more natural to state the result with the Ricci assumption. 
\end{rem}

\begin{rem}\label{volume ratio and Ricci bdd imply entropy lower bdd} The initial Ricci curvature lower bound in fact gives a Sobolev inequality on the geodesic balls in $M$, which in turn implies a Log Sobolev inequality and provides a lower bound for the local entropy in terms of its volume. 
Hence the local entropy $\nu(B_{g_0}(p,5),g_0,1)$ lower bound condition in the above theorem can be replaced by a uniform volume lower bound condition for the geodesic balls on $M$, namely, 
\be
\mathrm{Vol}_{g_0}(B_{g_0}(p,5))\geq v_0,
\ee
for some positive constant $v_0$, for all $p$ $\in M$. In that case, the constants $\e$, $C$ and $\hat T$ also depend on $v_0$. In particular, the global entropy $\nu(M,g)$'s lower bound can be replaced by $\Ric\geq 0$ and a lower bound on the  asymptotic volume ratio. 
\end{rem}

There have been many studies of compactness under scale-invariant integral curvature bounds, notably Anderson--Cheeger's diffeomorphism finiteness result \cite{AndersonCheeger1991}. Orbifold compactness results under $\|\Rm\|_{L^{n/2}}$ bounds have also been obtained for Einstein manifolds as well as for both compact and noncompact gradient Ricci solitons \cite{Anderson1989,CaoSesum2007,HaslhoferMuller2011}. In comparison, Theorem \ref{Theorem:GH} does not impose such analytic conditions on $(M_i,g_i)$, but does require sufficient smallness of the local scale-invariant curvature concentration.

Theorem \ref{Theorem:GH} is also a smoothing result for limit spaces of manifolds with lower curvature bounds, achieved via distance distortion estimates and pseudolocality-type estimates of the Ricci flow. There has been much recent work in this direction in many different settings \cite{BCW2019,McLeodTopping2018,McLeodTopping2019,SimonTopping2016,SimonTopping2017,Lai2019,Hochard2019,LeeTam2019,HuangRongWang-2020,HuangWang-2020}.

The structure of the rest of this paper is as follows. In Section \ref{CurvEst}, we prove our main smoothing result, Theorem \ref{pseudo}. In Section \ref{gapsection}, we prove our gap result for gradient Ricci soliton, Theorem \ref{gap thm for soliton}. In Section \ref{gapriccisection}, we prove our gap result for complete noncompact Ricci nonnegative manifolds, Corollary \ref{Corollary-diff}. Finally in Section \ref{ghsection}, we prove our Gromov--Hausdorff compactness result, Theorem \ref{Theorem:GH}.

{\it Acknowledgements}: The authors would like to thank Peter Topping for the interest in this work as well as Gilles Carron for pointing out the reference \cite{Cheeger2003} and related results. The authors would also like to thank the referee for pointing out a mistake in the earlier version of this paper. P.-Y. Chan would like to thank Bennett Chow, Lei Ni and Jiaping Wang for continuous encouragement and support. E. Chen thanks Guofang Wei and Rugang Ye for their support and was partially supported by an AMS--Simons Travel Grant. M.-C. Lee was partially supported by NSF grant DMS-1709894 and EPSRC grant number P/T019824/1. 

\section{Curvature estimates of Ricci flows}\label{CurvEst}
In this section, we will prove the semi-local estimates of the Ricci flow. We begin by fixing some notation below.

Suppose $(M^n,g)$  is an $n$ dimensional complete (not necessarily compact) Riemannian manifold and $\Omega$ is a connected domain on $M$ with smooth boundary (boundaryless if $M=\Omega$). Hereinafter, we shall reserve the positive integer $n$ for the dimension of $M$. Wang \cite{Wang-2018} localized Perelman's entropy and proved an almost monotonicity in local entropy when $\Omega$ is bounded, generalizing the result in \cite{Perelman-2002}. Using his notation, we have:
\begin{align}
&D_g(\Omega):=\left\{u: u\in W^{1,2}_0(\Omega), u\geq 0 \text{  and  } \|u\|_2=1 \right\},
\\
&W(\Omega, g, u, \tau):=\int_{\Omega}\tau(Ru^2+4|\nabla u|^2)-2u^2\log u    d\mu\label{def of local nu}
\\
&\qquad\qquad\qquad\qquad-\frac{n}{2}\log(4\pi\tau)-n,\notag
\\
&\nu(\Omega, g, \tau):= \inf_{u\in D_g(\Omega), s \in (0, \tau]}W(\Omega, g, u, s),
\\
&\nu(\Omega, g):= \inf_{\tau\in (0,\infty)}\nu(\Omega, g, \tau).
\end{align}

In order to prove the curvature estimate of Theorem \ref{pseudo}, we first show that it can be reduced to the preservation of local $L^{n/2}$ control of $\Rm(g(t))$. 

\begin{prop}\label{curv-esti}
For all $A>0$, there is $c_0(n, A)>0$ such that the following holds. 
Suppose $(M,g(t)),t\in [0,T]$ is a complete Ricci flow with bounded curvature such that for all $(x,t)\in M\times [0,T]$, the following holds:
\begin{enumerate}
    \item[(a)] $\nu(B_{g_0}(x,10^6n\sqrt{T}),g_0,2T)\geq -A$;
    \item[(b)] $\left(\int_{B_{g(t)}(x,\sqrt{t})}|\Rm(g(t))|^{n/2}d\mu_{g(t)}\right)^{n/2}\leq c_0 \e$ for $\e<1$,
\end{enumerate}
then we have 
\begin{equation}
\sup_M |\Rm(x,t)|< \e t^{-1}
\end{equation}
for all $t\in (0,T]$.
\end{prop}
\begin{proof}
By rescaling, we may assume $T=1$. Suppose on the contrary that the result is not true. Then for some $A,\e >0$, we can find sequences of $\delta_i=c_i\e_i$ with $c_i\rightarrow 0$, $\e_i$ $\in (0,1)$ and $\{(M_i,g_i(t),p_i)\}$ with bounded curvature such that
\begin{enumerate}
    \item $\nu(B_{g_i(0)}(x,10^6n),g_i(0),2)\geq -A$;
    \item  $\left(\int_{B_t(x,\sqrt{t})} |\Rm(g_i(t))|^{n/2} d\mu_{i,t}\right)^{2/n}\leq \delta_i \rightarrow 0$ for all $(x,t)\in M_i\times [0,1]$
\end{enumerate}
 but for some $(x_i,t_i)\in M_i\times (0,1]$,
$$|\Rm_i(x_i,t_i)|=\e_i t_i^{-1}.$$

We may choose $t_i>0$ such that for all $(y,s)\in M_i\times (0,t_i)$, 
\begin{equation}
|\Rm_i(y,s)|<2\e_i s^{-1}. 
\end{equation} 
This can be done since the upper bound of curvature vary continuously by boundedness of curvature. Let $Q_i=t_i^{-1}\geq 1$. Consider the rescaled Ricci flow $\tilde g_i(t)=Q_i g_i(Q_i^{-1}t)$ for $t\in [0, 1]$ which satisfies 
\begin{enumerate}
\item[(a)] $\nu(B_{g_i(0)}(y,10^6n),\tilde g_i(0),2)\geq  -A$ for all $y\in M_i$;
\item[(b)] $\left(\int_{B_{\tilde g_i(t)}(y,\sqrt{t})}|\Rm(\tilde g_i(t))|^{n/2} d\tilde\mu_{i,t}\right)^{2/n}\leq \delta_i\rightarrow 0$ for all $(y,t)\in M_i\times [0,1]$;
\item[(c)] $|\Rm_{\tilde g_i}(y,s)|<  s^{-1}$ on $M_i\times (0, 1)$;
\item[(d)] $|\Rm_{\tilde g_i}(x_i,1)|=  \e_i$.
\end{enumerate}
By (a) and \cite[Theorem 5.4]{Wang-2018}, we have uniform lower bound of the entropy $\nu(B_{\tilde{g}_i(t)}(y,1),\tilde g_i(t),1)$. Now (c) and \cite[Theorem 3.3]{Wang-2018} implies an uniform lower bound of the volume of $B_{\tilde g_i(t)}(x_i, r)$ which depends only on $A$ and $n$ for any $r$, $t$ $\in [1/2,1]$. By the curvature bound (c), we also have uniform Sobolev inequality on $B_{\tilde g_i(t)}(x_i, 1)$, for 
$t\in [1/2,1]$ (see \cite{Saloff-Coste-1992} and \cite{PLi-2012}). It follows from Kato's inequality and the evolution equation of $\Rm$ under the Ricci flow that 
\be
\heat |\Rm| \leq 8|\Rm|^2.
\ee

Since the curvature is uniformly bounded on $[\frac12,1]$, we may apply the Moser iteration argument \cite{PLi-2012}, together with (b) and H\"{o}lder inequality  to show that 
\be
\begin{split}
\e_i&=|\Rm_{\tilde g_i}(x_i,1)|\\
&\leq C(n, A)\int^1_{1/2}\fint_{B_{\tilde g_i(t)}(x_i, {1/2})} |\Rm_{\tilde g_i}| d\mu_s ds\\
&\leq C'(n, A)\left(\int^1_{1/2}\fint_{B_{\tilde g_i(t)}(x_i, {1/2})} |\Rm_{\tilde g_i}|^{n/2} d\mu_s ds \right)^{2/n}\\
&\leq C'(n, A)  c_i \e_i  \to 0 \text{  as  } i\to \infty,
\end{split}
\ee
which is impossible if $c_i$ is too small. This completes the proof of the lemma.
\end{proof}

\begin{rem}
 The a-priori curvature boundedness and completeness is in fact unnecessary, see Section~\ref{gapsection}.
\end{rem}

Next, we will show that if the initial local $L^{n/2}$ is small enough, then it is preserved in some semi-local sense. We first begin with the energy evolution of $L^{n/2}$ norm. 
\begin{lma}\label{Lemma:Lp pres}Suppose $n\geq 3$ and $(M,g(t))$ is a complete solution to the Ricci flow, $t$ $\in [0, T]$. Then for any $\a\geq \frac n4$, $\b >0$ and $\phi(x,t)$ compactly supported function in spacetime, there exist positive constants $C(\a)$ and $C'(n,\a)$ such that
\begin{equation}\label{lemma21ineq}
    \begin{split}
        \frac{d}{dt}\int_M \phi^2(|\Rm|^{2}+\b)^{\a}d\mu_t \leq& -C(\a)\int_M|\nabla (\phi (|\Rm|^2+\b)^{\a /2})|^2d\mu_t\\
& +C'(n,\a)\int_M \phi^2(|\Rm|^{2}+\b)^{\a+1/2}d\mu_t\\
& +C'(n,\a)\int_M  |\nabla \phi|^2(|\Rm|^{2}+\b)^{\a}d\mu_t.\\
&+\int_M 2\phi\Box \phi (|\Rm|^{2}+\b)^{\a}d\mu_t,
    \end{split}
\end{equation}
where $\Box= \frac{\p}{\p t}-\Delta_{g(t)}$.
\end{lma}
\begin{proof}We compute the time derivative of the integral norm as in \cite{Chen-2019}. Using $\frac{\p}{\p t}d\mu_t=-R d\mu_t\leq c(n)|\Rm| d\mu_t$, we have 

\begin{equation}
    \begin{split}
        \frac{d}{dt}\int_M \phi^2(|\Rm|^{2}+\b)^{\a}d\mu_t \leq& \int_M \frac{\p}{\p t} \left(\phi^2(|\Rm|^{2}+\b)^{\a}\right)d\mu_t  \\
&+ c(n)\int_M \phi^2(|\Rm|^{2}+\b)^{\a +1/2}d\mu_t .
    \end{split}
\end{equation}
For the first term on the R.H.S.,
\begin{equation}
    \begin{split}
&\quad  \int_M \frac{\p}{\p t} \left(\phi^2(|\Rm|^{2}+\b)^{\a}\right)d\mu_t\\
=&  \int_M \Box \left(\phi^2(|\Rm|^{2}+\b)^{\a}\right)d\mu_t\\
=&  \int_M 2\phi\Box \phi (|\Rm|^{2}+\b)^{\a}d\mu_t- \int_M 2|\nabla \phi|^2 (|\Rm|^{2}+\b)^{\a}d\mu_t\\
&+ \int_M \a\phi^2(|\Rm|^{2}+\b)^{\a-1}\Box |\Rm|^2d\mu_t\\
&- \int_M 4\a(\a-1)\phi^2(|\Rm|^{2}+\b)^{\a-2}|\Rm|^2|\nabla |\Rm||^2d\mu_t\\
&-\int_M 8\a\phi\la \nabla \phi, \nabla |\Rm|\ra |\Rm|(|\Rm|^{2}+\b)^{\a-1}d\mu_t,
    \end{split}
\end{equation}
where $\Box =\frac{\p}{\p t}-\Delta_{g(t)}$. To proceed, we apply the evolution equation of $|\Rm|^2$ (see \cite{Chen-2019} and ref. therein) 
\be\label{evolution of Rm sq}
\Box |\Rm|^2\leq -2|\nabla \Rm|^2+16 |\Rm|^3.
\ee
It follows from \eqref{evolution of Rm sq} that
\begin{equation}
    \begin{split}
 \int_M \a\phi^2(|\Rm|^{2}+\b)^{\a-1}\Box |\Rm|^2 d\mu_t\leq& -2\a\int_M \phi^2(|\Rm|^{2}+\b)^{\a-1}|\nabla \Rm|^2d\mu_t\\
&+ 16\a  \int_M \phi^2(|\Rm|^{2}+\b)^{\a +1/2}d\mu_t.
    \end{split}
\end{equation}
Hence by Kato's inequality and H\"{o}lder inequality,
\begin{eqnarray*}
&& \int_M \a\phi^2(|\Rm|^{2}+\b)^{\a-1}\Box |\Rm|^2 d\mu_t\\
&&- \int_M 4\a(\a-1)\phi^2(|\Rm|^{2}+\b)^{\a-2}|\Rm|^2|\nabla |\Rm||^2d\mu_t\\
&&-\int_M 8\a\phi\la \nabla \phi, \nabla |\Rm|\ra |\Rm|(|\Rm|^{2}+\b)^{\a-1}d\mu_t\\
&\leq& -C(\a)\int_M \phi^2(|\Rm|^{2}+\b)^{\a-1}|\nabla |\Rm||^2d\mu_t\\
&&+ 16\a  \int_M \phi^2(|\Rm|^{2}+\b)^{\a +1/2}d\mu_t\\
&&+\int_M 8\a\phi |\nabla \phi| |\nabla |\Rm|| |\Rm|(|\Rm|^{2}+\b)^{\a-1}d\mu_t\\
&\leq& -C'(\a)\int_M \phi^2(|\Rm|^{2}+\b)^{\a-2}|\Rm|^2|\nabla |\Rm||^2d\mu_t\\
&&+ 16\a  \int_M \phi^2(|\Rm|^{2}+\b)^{\a +1/2}d\mu_t\\
&&+C''(\a)\int_M |\nabla \phi|^2(|\Rm|^{2}+\b)^{\a}d\mu_t.
\end{eqnarray*}
We also have by Cauchy Schwarz inequality
\begin{equation}
    \begin{split}
&\quad C'(\a)\int_M \phi^2(|\Rm|^{2}+\b)^{\a-2}|\Rm|^2|\nabla |\Rm||^2d\mu_t\\
&= C'''(\a) \int_M \phi^2 |\nabla (|\Rm|^{2}+\b)^{\frac{\a}{2}} |^2d\mu_t\\
&\geq \frac{C'''(\a)}{2}\int_M  |\nabla (\phi (|\Rm|^{2}+\b)^{\frac{\a}{2}}) |^2d\mu_t- C'''(\a)\int_M |\nabla \phi|^2(|\Rm|^{2}+\b)^{\a}d\mu_t.
    \end{split}
\end{equation}

All in all,
\begin{equation}
    \begin{split}
        \frac{d}{dt}\int_M \phi^2(|\Rm|^{2}+\b)^{\a}d\mu_t 
&\leq C(n, \a)\int_M \phi^2(|\Rm|^{2}+\b)^{\a +1/2}d\mu_t \\
&\quad +  \int_M 2\phi\Box \phi (|\Rm|^{2}+\b)^{\a}d\mu_t\\
&\quad+ C(\a)\int_M 2|\nabla \phi|^2 (|\Rm|^{2}+\b)^{\a}d\mu_t\\
&\quad -\frac{C'''(\a)}{2}\int_M  |\nabla (\phi (|\Rm|^{2}+\b)^{\frac{\a}{2}}) |^2d\mu_t.
    \end{split}
\end{equation}
Our desired inequality \eqref{lemma21ineq} then follows.
\end{proof}

We also need the following lemma showing that the local entropy implies local Sobolev inequality. 
\begin{lma}\label{Lemma-Sobo}
For all $A\geq 1000n,\lambda>0$ and $\delta>0$, there are positive constants $C_S(n,A,\lambda, \delta)$ and $\hat T(n,A,\lambda, \delta)$ such that the following holds. Suppose $(M,g(t))$ is a complete Ricci flow with bounded curvature on $[0,T]$ and for all $p\in M$ and all $t\in (0,T]$, all of the following conditions are satisfied
\begin{enumerate}
\item $R_{g(0)}\geq -n\lambda$;
\item $\nu(B_0(p, 5), g(0),1)\geq -A$;
\item $\Ric(p, t)\leq \delta t^{-1}$.
\end{enumerate}
Then we have for any $p$ $\in M$ and  $t$ $\in (0,  \min\{T,\hat T\}]$,
\be\label{entropy lower bound preserved 2}
\nu(B_t(p, 2), g(t), 32^{-1})\geq -2A
\ee
and for any $\varphi$ $\in C^{\infty}_0(B_t(p, 2))$
\be\label{uniform Sobolev inequality}
\left(\int_{B_t(p, 2)}|\varphi|^{\frac{2n}{n-2}} d\mu_t \right)^{\frac{n-2}{n}}\leq C_S \left( \int_{B_t(p, 2)} |\nabla \varphi|^2+ (R+c_n\lambda +1) \varphi^2 d\mu_t \right).
\ee
\end{lma}
\begin{proof}
For \eqref{entropy lower bound preserved 2}, we apply \cite[Theorem 5.4]{Wang-2018} to get for all small $t\leq \min\{ \hat T(n,A, \delta), T\}$
\be
\begin{split}
\nu(B_t(p, 2), g(t), 32^{-1})&\geq \nu(B_0(p, 5), g(0), t+32^{-1})-16t\\
&\geq  \nu(B_0(p, 5), g(0), 1)-A\\
&\geq -2A
\end{split}
\ee
This completes the proof of \eqref{entropy lower bound preserved 2}. Since the Ricci flow has bounded curvature, it follows from the maximum principle that there exists a dimensional constant $c_n$ such that for all $t\leq \min\{ \hat T(n,A, \delta), T\}$
\be
R(x,t)\geq-c_n\lambda. 
\ee
By the  definition of local $\nu$ entropy \eqref{def of local nu} and \eqref{entropy lower bound preserved 2}, we have a uniform Log Sobolev inequality: for any $\tau$ $\in (0,32^{-1})$, $u \in W^{1,2}_0(B_t(p, 2))$, with $\|u\|_{g(t), 2}=1$,
\be\label{uniform log Sobolev inequality from entropy}
\begin{split}
\int_{B_t(p, 2)}u^2\log u^2    d\mu_t\leq&  \tau \int_{B_t(p, 2)}4|\nabla u|^2+Ru^2 d\mu_t\\
&-\frac{n}{2}\log(4\pi\tau)-n+2A.
\end{split}
\ee
The uniform Log Sobolev inequality then implies a uniform Sobolev inequality along the Ricci flow as first described in \cite{Zhang07} (see also \cite{Chen-2019,Ye-2015} and Theorems \ref{sgp L2L1 est} and \ref{L2 sgp est to Sob ineq}).
Indeed, when ${\p B_t(p, 2)}$ is nonempty, the same arguments as in \cite{Davies-1989, Ye-2015} will give us Theorems \ref{sgp L2L1 est} and \ref{L2 sgp est to Sob ineq} below for the Dirichlet Sobolev inequality, and these together with \eqref{uniform log Sobolev inequality from entropy} imply \eqref{uniform Sobolev inequality}, finishing the proof.
\end{proof}

 We shall now state Theorems \ref{sgp L2L1 est} and \ref{L2 sgp est to Sob ineq} without mentioning the proofs, since they are essentially the same as those found in \cite{Davies-1989,Ye-2015}. Let $(N, h)$ be a smooth compact Riemannian manifold with metric $h$ and smooth boundary $\p N$, $H$ the elliptic operator $=-\Delta +4^{-1}(R+c_n\lambda)$, where $\lambda$ and $c_n$  are non-negative constants such that $R\geq -c_n\lambda$ on $N$,
\be
Q(u, v)= \int_{N} \nabla u\cdot \nabla v+ 4^{-1}(R+c_n\lambda) u\cdot v d\mu_h
\ee
 and write $Q(u, u)$ as $Q(u)$. For $t>0$, consider the semigroup $e^{-tH}$ of the operator $H$. For any $u_0\in L^2(N, h)$, the function $u(t):=e^{-tH}u_0$ is the solution to the Dirichlet evolution equation
\be\label{heat equation for H}
\left\{
\begin{array}{ll}
     &  \frac{\p u}{\p t}=-H u\\
     & u(0)= u_0\\
     &u= 0 \text{  on } \p N.
\end{array}
\right.
\ee

\begin{thm}[\cite{Davies-1989, Ye-2015}]\label{sgp L2L1 est}
Let $\sigma^*$ $\in (0,\infty]$. Suppose that for all $\sigma$ $\in (0,\sigma^*)$,
\be
\int_N u^2\log u^2 d\mu\leq \sigma\int_N |\nabla u|^2+4^{-1}(R+c_n \lambda)u^2 d\mu+ \b(\sigma)
\ee\label{log Sob start}
is true for any $u$ $\in W^{1,2}_0(N)$ with $\|u\|_2=1$, where $\b$ is a continuous nonincreasing function and $R+c_n \lambda\geq 0$. If in addition the function,
\be
\tau(t):=\frac{1}{2t}\int_0^t\b(s) ds
\ee
is finite for any $t$ $\in (0, \sigma^*)$. Then for each $u$ $\in L^2(N)$
\be\label{semigp L2 est}
\|e^{-tH}u\|_{\infty}\leq e^{\tau(t)}\|u\|_2
\ee
for $t \in (0,\sigma^*/4)$. Moreover, for all  $u$ $\in L^1(N)$
\be\label{semigp L1 est}
\|e^{-tH}u\|_{\infty}\leq e^{2\tau(\frac{t}{2})}\|u\|_1
\ee
for $t \in (0,\sigma^*/4)$. 
\end{thm}


\begin{thm}[\cite{Davies-1989, Ye-2015}]\label{L2 sgp est to Sob ineq} Suppose there exist positive constants $c_1$ and $t_1$ such that for all $t$ $\in (0, t_1)$ and $u$ $\in L^2(N)$
\be\label{assump finite t}
\|e^{-tH}u\|_{\infty}\leq c_1t^{-\frac{n}{4}}\|u\|_2.
\ee
Set $H_0=H+1$. Then for some constant $C(c_1, t_1 , n)$,
\be
\|H_0^{-1/2}u\|_{\frac{2n}{n-2}}\leq C\|u\|_2,
\ee
for any  $u$ $\in L^2(N)$. In particular, 
\be
\begin{split}
\|u\|_{\frac{2n}{n-2}}^2&\leq C^2\|H^{1/2}_0 u\|_2^2\\
&\leq C^2(Q(u)+\|u\|_2^2).
\end{split}
\ee
for all $u$ $\in W^{1, 2}_0(N)$, where $H_0^{1/2}$ and $H_0^{-1/2}$ denote the fractional operator of $H_0$ and its inverse respectively (see \cite{Davies-1989, Ye-2015}).
\end{thm}

With Proposition \ref{Lemma:Lp pres} and Lemmas \ref{curv-esti} and \ref{Lemma-Sobo} in hand, we can now prove Theorem~\ref{pseudo} to conclude this section.
\begin{proof}[Proof of Theorem~\ref{pseudo}]
Let $\Lambda$ be a constant to be chosen later. Since $g(t)$ has bounded curvature, we may let $\hat T$ be the maximal time such that for all $(x,t)\in M\times [0,\hat T\wedge T)$, we have
\begin{eqnarray}
\left(\int_{B_{g(t)}(x,1)}|\Rm(g(t))|^{n/2}d\mu_t\right)^{2/n} \leq  \Lambda \e. 
\end{eqnarray}
By \cite[Theorem 3.3]{Wang-2018}, Proposition~\ref{curv-esti} and the injectivity radius estimates in \cite{CheegerGromovTaylor-1982}, it suffices to show that $\hat T$ is bounded from below uniformly if $\e$ and $\Lambda$ are chosen appropriately. 

We may choose $\e$ small enough so that $\Lambda \e=\delta<1$ is small. Therefore, the Ricci flow $g(t)$ satisfies
\begin{enumerate}
    \item $\left(\int_{B_{g_0}(x,2)}|\Rm(g(0))|^{n/2}d\mu_0\right)^{2/n} \leq  \e$;
    \item $\left(\int_{B_{g(t)}(x,1)}|\Rm(g(t))|^{n/2}d\mu_t\right)^{2/n} \leq \delta$;
    \item $\nu(B_{g_0}(x,5),g_0,1)\geq -A$;
    \item $\Ric(g_0)\geq -\lambda$
\end{enumerate}
for all $(x,t)\in M\times [0,T\wedge \hat T)$. By the monotonicity of the local entropy over domain and scale, we may assume $\hat T<25n^{-2}10^{-12}$ so that applying Proposition~\ref{curv-esti} on $g(t),t\in [0,T\wedge \hat T)$ yields 
\begin{equation}\label{cur-apr}
    \sup_M |\Rm(g(t))|\leq c(n,A)\delta t^{-1}<t^{-1}.
\end{equation}

Now we are ready to estimate $\hat T$. For any $x_0\in M$, we let $\eta(x,t)=d_t(x,x_0)+c_n\sqrt{t}$ and define $\phi(x,t)=e^{-10t}\varphi(\eta(x,t))$ where $\varphi(s)$ is a cutoff function on $\mathbb{R}$ so that $\varphi\equiv 1$ on $(-\infty,\frac12]$, $\varphi\equiv 0$ outside $(-\infty,1]$ and satisfies $\varphi''\geq -10\varphi $, $0\geq \varphi'\geq -10\sqrt{\varphi}$. By choosing $c_n$ large enough, we have from \cite[Lemma 8.3]{Perelman-2002} that  
\begin{eqnarray}
\heat \phi \leq 0
\end{eqnarray}
in the sense of barrier and hence in the sense of distribution, see \cite[Appendix A]{Mantegazza2014}.

Using Lemma~\ref{Lemma:Lp pres} with the above choice of $\phi$ and $\alpha=n/4$, we conclude that 
\begin{equation}
    \begin{split}
           \frac{d}{dt}\int_M \phi^2(|\Rm|^{2}+\b)^{n/4}d\mu_t \leq& -C_n^{-1}\int_M|\nabla (\phi (|\Rm|^2+\b)^{n/8})|^2d\mu_t\\
& +C_n\int_M \phi^2(|\Rm|^{2}+\b)^{n/4+1/2}d\mu_t\\
& +C_n\int_{supp(\phi)}  (|\Rm|^{2}+\b)^{n/4}d\mu_t.
\end{split}
\end{equation}

Noted that $\phi$ is supported on $B_t(x_0,1)$, 
By Lemma~\ref{Lemma-Sobo}, the first term on the right can be estimated as 
\begin{equation}
    \begin{split}
       &\quad  C_n^{-1}\int_M|\nabla (\phi (|\Rm|^2+\b)^{n/8})|^2d\mu_t \\
        &\geq  C_1(n,A) \left( \int_M |(\phi^2 (|\Rm|^2+\b)^{n/4})|^\frac{n}{n-2}d\mu_t  \right)^\frac{n-2}{n}\\
        &\quad - C_1(n,A)\int_{M}\phi^2 (R+c_n\lambda+1) (|\Rm|^2+\beta)^{n/4} d\mu_t
    \end{split}
\end{equation}
while the second term can be estimated by 
\begin{equation}
    \begin{split}
        &\quad C_n\int_M \phi^2(|\Rm|^{2}+\b)^{n/4+1/2}d\mu_t\\
        &\leq C_n\left(\int_{supp(\phi )} (|\Rm|^2+\beta)^\frac{n}{4} d\mu_t \right)^\frac{2}{n}\left(\int_M \left[\phi^2 (|\Rm|^2+\beta)^\frac{n}{4} \right]^\frac{n}{n-2} d\mu_t  \right)^\frac{n-2}{n}\\
        &\leq C_n \delta\left(\int_M \left[\phi^2 (|\Rm|^2+\beta)^\frac{n}{4} \right]^\frac{n}{n-2} d\mu_t  \right)^\frac{n-2}{n}
    \end{split}
\end{equation}
as $\beta\rightarrow 0$. We can apply the same argument to $\int_M \phi^2 R (|\Rm|^2+\beta)^{n/4} d\mu_t$ to deduce the same upper bound. Therefore, we conclude that if $\delta\leq \sigma(n,A)<<1$, then as $\beta\rightarrow 0$ we have 
\begin{equation}
    \begin{split}
           \frac{d}{dt}\int_M \phi^2(|\Rm|^{2}+\b)^{n/4}d\mu_t \leq
& C(n,A, \lambda)\int_{B_t(x_0,1)}  (|\Rm|^{2}+\b)^{n/4}d\mu_t\\
\leq & C(n,A, \lambda) \delta^{\frac{n}{2}}
\end{split}
\end{equation}

By letting $\beta\to 0$ together with the assumption on the initial metric, we conclude that for all $(x_0,t)\in M\times [0,T\wedge \hat T)$,
\begin{eqnarray}
\int_{B_t(x_0,\frac14)}|\Rm|^{n/2} d\mu_t \leq e^{10\hat T}(C_1(n,A, \lambda) \Lambda^{\frac{n}{2}} t +1) \e^{\frac{n}{2}}. 
\end{eqnarray}

Now we claim that there is $\tilde T(n,A)$ depending only on $n,A$ such that for all $(y,t)\in M\times [0,T\wedge \hat T]$, if $T\wedge \hat T\leq \tilde T(n,A)$, then we have 
\begin{eqnarray}\label{ball-inc}
B_t(y,1)\subset \bigcup_{i=1}^N B_t(x_i,\frac14)
\end{eqnarray}
for some $N(n,A,\lambda)\in \mathbb{N}$. If the claim is true, then we conclude that  for all $(y,t)\in M\times [0,T\wedge \hat T\wedge \tilde T(n,A)]$, 
\begin{equation}
    \begin{split}
        \int_{B_t(y,1 )}|\Rm|^{n/2} d\mu_t 
        &\leq  \sum_{i=1}^N \int_{B_t(x_i,\frac14 )}|\Rm|^{n/2} d\mu_t \\
        &\leq Ne^{10\tilde T} (C_1 \Lambda^{\frac{n}{2}} \tilde T +1) \e^{\frac{n}{2}}. 
    \end{split}
\end{equation}

Therefore if we choose $\Lambda^{n/2}=4N$ and further require $\tilde T\leq (4^{\frac{n}{2}}N^{\frac{n}{2}}C_1)^{-1}$, then we have contradiction and hence $\hat T\geq \tilde T(n,A,\lambda)$. This will complete the proof. Hence, it remains to establish the uniform covering. 

For each $(y,t)\times M\times [0,T\wedge \hat T\wedge \tilde T)$, we let $\{x_i\}_{i=1}^N$ be a maximal set of points in $B_t(y,1)$ such that $B_t(x_i,\frac18)$ are  disjoint from each other and \eqref{ball-inc} holds. By \eqref{cur-apr} and distance distortion estimates \cite[Lemma 8.3]{Perelman-2002}, we have $B_t(y,1)\subset B_0(y,2)$ if $\tilde T$ is small. At the same time, by choosing $\delta$ sufficiently small, we may apply the proof of \cite[Lemma 2.4]{HuangKongRongXu2018} (see also \cite[Lemma 2.2]{LeeTam2020}) to show that $B_t(x_i,\frac18)\supset B_0(x_i,r_0)$ for some uniformly small $r_0$. Therefore,
\begin{equation}
    \begin{split}
       \sum_{i=1}^N \mathrm{Vol}_{g_0}\left(B_0(x_i,r_0)\right) &\leq \sum_{i=1}^N \mathrm{Vol}_{g_0}\left(B_t(x_i,\frac18)\right)\\
        &\leq \mathrm{Vol}_{g_0}\left(B_t(y,1)\right)\\
        &\leq \mathrm{Vol}_{g_0}\left(B_0(y,2)\right).
    \end{split}
\end{equation}
Since $x_i\in B_t(y,1)\subset B_0(y,2)$, the estimates on $N$ then follows from $\Ric(g_0)$ lower bound and volume comparison. The desired result follows by re-labelling the constants. 
\end{proof}
\section{Gap Theorem of Ricci solitons}\label{gapsection}
In this section we will prove Theorem \ref{gap thm for soliton}, a gap theorem for shrinking and steady gradient Ricci solitons. We do not assume a-priori bounds on the curvature. The novel idea is to  obtain local curvature control under the small $L^{n/2}$ curvature and local entropy bound (see also \cite{GeJiang-2017}). We first prove the following result, from which Theorem \ref{gap thm for soliton} shall follow.
\begin{thm}
\label{local curvature}
For all $A\geq 1000n$, there is $\e(n,A),C(n,A), \hat T(n,A)>0$ such that the following holds. Suppose $(M,g(t))$ is a Ricci flow on $[0,T]$ and $p\in M$ be a point such that for all $t\in (0,T]$,
\begin{enumerate}
\item $B_t(p,1)\Subset M$;
\item $\left(\int_{B_t(p,4A\sqrt{t})}|\Rm|^{n/2}d\mu_t\right)^{2/n}\leq \e_0$ for some $\e_0<\e$;
\item $\nu(B_t(p,4A\sqrt{t}),g(t),t)\geq -A$;
\end{enumerate}
then we have 
\begin{equation}
    \left\{
    \begin{array}{ll}
         &  |\Rm|(x,t)\leq C(n,A)\e_0 t^{-1}\\
         & \mathrm{inj}(x,t)\geq C(n,A)^{-1}\sqrt{t}
    \end{array}
    \right.
\end{equation}
 for all $x\in B_t(p,\frac14 A\sqrt{t}),t\leq T\wedge \hat T$.
\end{thm}
\begin{proof}We will split the proof into three parts. 

{\bf Step 1. Rough estimates under stronger assumption.} We first prove the rough curvature estimate: $|\Rm(x,t)|\leq C(n,A)t^{-1}$ on $B_t(p,\frac12 A\sqrt{t})$ under an extra assumption: 
$$\star\quad \quad  \Ric(x,t)\leq t^{-1},\;\;\text{on}\;\; B_t(p,\sqrt{t}),t\in (0,T].$$

The injectivity radius estimates will follow from the work of \cite{CheegerGromovTaylor-1982} by \cite[Theorem 3.3]{Wang-2018}. For $x\in B_t(p,A\sqrt{t}), t<T$, we define 
\begin{equation}
r(x,t)=\sup \left\{0<r<A\sqrt{t}-d_t(x,p): \sup_{P(x,t,r)}|\Rm|\leq r^{-2} \right\}
\end{equation}
where $P(x,t,r)=\{ (y,s): y\in B_s(x,r),s\in [t-r^2,t]\cap (0,T]\}$. We claim that there are $\e,c_0,T_0>0$ depending only on $n,A$ such that if assumptions  hold for $\e_0<\e$, then for all $x\in B_t(p,A\sqrt{t})$ and $t<T\wedge T_0$,
\begin{equation}
F(x,t)=\frac{r(x,t)}{A\sqrt{t}-d_t(x,p)} \geq c_0.
\end{equation}
The rough curvature estimate then follows immediately from the claim since for any $x$ $\in B_t(p,\frac12 A\sqrt{t})$ and $t<T\wedge T_0$,
\be
\begin{split}
|\Rm|(x,t)&\leq \frac{1}{r^2(x, t)}\\
&\leq \frac{1}{c_0^2(A\sqrt{t}-d_t(p,x))^2}\\
&\leq \frac{4}{c_0^2A^2t}.
\end{split}
\ee

{ 

Suppose on the contrary that the claim is not true for some $A$ and $n$, we can find a sequence of Ricci flow $\{(M_i,g_i(t),p_i)\}_{i=1}^\infty$, $t_i,\e_i\rightarrow 0$ such that 
\begin{itemize}
\item $\Ric_i(x,t)< t^{-1}$ for all  $x\in B_t(p_i,\sqrt{t})$, $t<t_i$;
\item $\displaystyle\int_{B_{g_i(t)}(p_i,4A\sqrt{t})}|\Rm_i|^{n/2} d\mu_{i,t}\leq  \e_i$ for all $t<t_i$;
\item $\nu(B_{g_i(t)}(p,4A\sqrt{t}),g_i(t),t)\geq -2A$ for $t<t_i$,
\end{itemize}
but for some sequence $x_i\in B_t(p_i,A\sqrt{t})$, we have
\begin{equation}\label{lower bdd of Fi}
F_i(x_i,t_i)=\min \{F_i(y,s): s\in (0,t_i), y\in B_s(p_i,A\sqrt{s})\} \rightarrow 0.
\end{equation}
Re-scale the flow by $\tilde g_i(t)=Q_i g_i(t_i+Q_i^{-1}t)$, $-Q_it_i\leq t\leq 0$ where $Q_i^{-1/2}=r_i(x_i,t_i)$ so that $\tilde r_i(x_i,0)=1$. Then by \eqref{lower bdd of Fi}
\begin{equation}
\begin{split}
d_{\tilde g_i(0)}(x_i,\partial B_{ g_i(t_i)}(p_i,A\sqrt{t_i}))
&=\frac{d_{ g_i(t_i)}(x_i,\partial B_{ g_i(t_i)}(p_i,A\sqrt{t_i}))}{r(x_i,t_i)}\\
&\geq \frac{A\sqrt{t_i}-d_{t_i}(x_i,p_i)}{r(x_i,t_i)}\\
&=F_i(x_i,t_i)^{-1}\\
&\rightarrow +\infty.
\end{split}
\end{equation}
That is to say the pointed Cheeger-Gromov limit of the flow centred at $x_i$ is complete provided it exists. Furthermore, we may invoke \eqref{lower bdd of Fi} again to see that
\begin{equation}\label{QT to infty}
\begin{split}
Q_it_i&=\frac{t_i}{r_i(x_i,t_i)^2}\\
&> \frac{1}{A^2}\left(\frac{A\sqrt{t_i}-d_{t_i}(x_i,p_i)}{r_i(x_i,t_i)}\right)^2\\
&\rightarrow +\infty.
\end{split}
\end{equation}
Next, we would like to show that after passing to a sub-sequence, the flows converge in Cheeger-Gromov sense. The two key ingredients are uniform curvature bound in $i$ on compact sets in spacetime and the injectivity radius lower bound at $x_i$ w.r.t. $\tilde g_i(0)$. 

  Let $r>0$ and $y\in \tilde B_{s}(x_i,r), s\in (-1,0]$. 
Using $r_i(x_i,t_i)=Q_i^{-1/2}<<\sqrt{t_i}$ and assumption $\star$, we apply Hamilton-Perelman's distance estimates \cite[Lemma 8.3]{Perelman-2002} (see also \cite{Hamilton-1995}) with $r_0=Q_i^{-1/2}$ and $K=Q_i$ so that 
\begin{equation}\label{dist distortion estimate}
\begin{split}
d_{Q_i^{-1}s+t_i}(x_i,p_i)&\leq d_{t_i}(x_i,p_i)+C_n Q_i^{1/2}\cdot (t_i-Q_i^{-1}s-t_i)\\
&\leq d_{t_i}(x_i,p_i)+C_n Q_i^{-1/2}.
\end{split}
\end{equation}

Hence
\be\label{dist y estimate}
\begin{split}
d_{Q_i^{-1}s+t_i}(y,p_i)&\leq Q_i^{-\frac 12}r+ d_{Q_i^{-1}s+t_i}(x_i,p_i)\\
                                     &\leq C_n Q_i^{-\frac 12}r+d_{t_i}(x_i,p_i)
\end{split}
\ee
 It follows from \eqref{lower bdd of Fi} and \eqref{QT to infty} that for all large $i>N(n, A, r)$, 
\be\label{dist y estimate 2}
\begin{split}
A\sqrt{Q_i t_i+s}-Q_i^{\frac{1}{2}}d_{t_i}(x_i,p_i) &\geq A\sqrt{Q_i t_i-1}-Q_i^{\frac 12}d_{t_i}(x_i,p_i)\\
                        &\geq A\sqrt{Q_it_i}-Q_i^{\frac 12}d_{t_i}(x_i,p_i)-A\\
&= F(x_i, t_i)^{-1}-A\\
&> C_nr.
\end{split}
\ee
Hence by  \eqref{dist y estimate}, $y\in B_{Q_i^{-1}s+t_i}(p_i, A\sqrt{Q_i^{-1}s+t_i})$. We have by  \eqref{lower bdd of Fi}, \eqref{dist y estimate} and \eqref{dist y estimate 2} that
\be
\begin{split}
\tilde r_i(y,s)&=\frac{r_i(y,Q_i^{-1}s+t_i)}{r_i(x_i,t_i)}\\
&=\frac{F_i(y,Q_i^{-1}s+t_i)}{F_i(x_i,t_i)}\cdot \frac{A\sqrt{Q_i^{-1}s+t_i}-d_{Q_i^{-1}s+t_i}(y,p_i)}{A\sqrt{t_i}-d_{t_i}(x_i,p_i)}\\
&\geq  \frac{A\sqrt{Q_i^{-1}s+t_i}-d_{Q_i^{-1}s+t_i}(y,p_i)}{A\sqrt{t_i}-d_{t_i}(x_i,p_i)}\\
&\geq \frac{A\sqrt{Q_i^{-1}s+t_i}-d_{Q_i^{-1}s+t_i}(x_i,p_i)-Q_i^{-1/2}r}{A\sqrt{t_i}-d_{t_i}(x_i,p_i)}\\
&\geq \frac{F(x_i, t_i)^{-1}-A-C_nr}{F(x_i, t_i)^{-1}}.
\end{split}
\ee
Thus for all $y\in \tilde B_s(x_i,r),s\in (-1,0]$, $r>0$ and $i>N(n, A, r)$, we have 
\begin{equation}\label{curvature radius lower bdd}
\tilde r_i(y,s)>\frac12.
\end{equation}

This gives the curvature estimates on any compact subset in space-time. By our assumptions, for any $r>0$ and $i>N(n,A, r)$ , the entropy satisfies

\be\label{entropy lower bdd 2}
\begin{split}&\quad \nu(\widetilde{B}_t(x_i, r),\tilde{g}_i(t),Q_it_i+t)\\&\geq
\nu(\widetilde{B}_t(p_i,4A\sqrt{Q_it_i+t}),\tilde{g}_i(t),Q_it_i+t)\\
&=\nu(B_t(p_i,4A\sqrt{t_i+Q_i^{-1}t}),g_i(t_i+Q_i^{-1}t), t_i+Q_i^{-1}t)\\
                                       &\geq -2A.\\
\end{split}
\ee

By virtue of \eqref{QT to infty}, \eqref{curvature radius lower bdd}, \eqref{entropy lower bdd 2} and \cite[Theorem 3.3]{Wang-2018}, the volume ratios $\frac{\widetilde V_0(\widetilde{B}_0(x_i, r))}{ r^n}$ are uniformly bounded from below in $i$ for any all $r$ $\in (0,1/2]$. Thanks to \eqref{curvature radius lower bdd} and Cheeger-Gromov-Taylor injectivity radius estimate \cite{CheegerGromovTaylor-1982}, the injectivity radius at $x_i$ w.r.t. $\tilde g_i(0)$ have a uniform positive lower bound in $i$. Hence by Hamilton's compactness theorem (see \cite{Hamilton-1995.2}, \cite{Chowetal-2007}), we can pass $\tilde g_i(t)$ to a complete limiting Ricci flow $(M_\infty,\tilde g_\infty(t),x_\infty),\, t\in (-1,0]$ which is a complete solution with bounded curvature. By the choice of $Q_i^{-1/2}$ and \eqref{dist y estimate}, we have $B_{t_i}(x_i,Q_i^{-1/2})\Subset B_{t_i}(p_i,A\sqrt{t_i})$. Therefore,
for all $ -1<s<0$, if $d_{Q_i^{-1}s+t_i}(x_i,y)<r$, then by \eqref{dist distortion estimate}
\be
\begin{split}
d_{Q_i^{-1}s+t_i}(y,p_i)&\leq d_{Q_i^{-1}s+t_i}(y,x_i)+d_{Q_i^{-1}s+t_i}(x_i,p_i)\\
&\leq r+d_{t_i}(x_i,p_i)+C_n Q_i^{-1/2}\\
&\leq  r+2 A\sqrt{t_i}.
\end{split}
\ee
Therefore, if $r= \sqrt{t_i}$, then we have $$B_{Q_i^{-1}s+t_i}(x_i,\sqrt{t_i})\Subset B_{Q_i^{-1}s+t_i}(p_i,3A\sqrt{Q_i^{-1}s+t_i})$$ for all $i\rightarrow +\infty$. This together with the assumption implies 
\begin{equation}\label{limit integral}
\int_{M_\infty} |\widetilde \Rm|^{n/2} d\tilde\mu_s \equiv 0
\end{equation}
for all $-1<s\leq 0$  and hence is a flat solution. Moreover, by the monotonicity of local entropy over domain, \eqref{entropy lower bdd 2} and the proof of Lemma 6.28 in \cite{Chowetal-2007}, we have 
\begin{equation}
\begin{split}\label{limit entropy lower bdd}
\nu(M_\infty,\tilde{g}_\infty(t))\geq -2A
\end{split}
\end{equation}
for all $-1<t\leq 0$. Recall that we have $\tilde r_i(x_i,0)=1$. 
As the entropy is bounded from below for all scales, the manifolds must be of maximum volume growth which implies that $\tilde{g}_\infty(t)$ is the static flat Euclidean metric. This contradicts with the curvature radius at $(x_\infty,0)$ and completes the proof under the assumption $\star$. Now the injectivity radius estimate follows from the curvature estimate and the work of \cite{CheegerGromovTaylor-1982}.
}
\bigskip 

{\bf Step 2. Removing assumption $\star$ in Step 1.} Since $B_t(p,1)\Subset M$ for $t\leq T$, by smoothness of solution we may find $\tilde T\leq T$ such that $|\Ric|<t^{-1}$ for $x\in B_t(p,\sqrt{t}),\; t\in (0,\tilde T)$.    W.L.O.G., we may assume that $\tilde T$ to be small uniformly, otherwise the required estimate on $|\Rm|$ follows by Step 1. 
Hence the result under $\star$ gives the curvature estimates over a smaller ball, i.e. for some $\hat T(n, A)$, 
\be\label{C0 Rm bdd in prop}
|\Rm|(x,t)\leq C(n,A)t^{-1}
\ee
for all $x\in B_t(p,\frac12 A\sqrt{t})$, $t\leq \min\{\tilde T, \hat T(n, A)\}$.

We claim that $\tilde T\geq T\wedge  \hat T(n, A)$. Suppose that is not the case, denote $s=\tilde T$, then by the maximality of $\tilde T$ there is $\bar x\in \overline{B_s(p,\sqrt{s})}$ such that $|\Ric|(\bar x,s)=s^{-1}$. By considering the flow $ s^{-1}g(st),t\in [0,1]$, we may wlog assume $s=1$. 
By the estimates of $\mathrm{inj}(x,t)$, \eqref{C0 Rm bdd in prop}, Theorem 3.3 in \cite{Wang-2018}, $\mathrm{Vol}_{g(s)}\left(B_s(\bar x,\frac14 A\sqrt{s})\right)$ 
is uniformly bounded from below  for any $s$ $\in [1/2, 1]$. Together with a result of Saloff-Coste \cite{Saloff-Coste-1992}, we get a uniform Sobolev inequality on $B_s(\bar x, \frac14A\sqrt{s})$ for any $s$ $\in [1/2, 1]$.
Then the Moser iteration argument \cite[Chapter 19]{PLi-2012} on $B_s(\bar x, \frac{1}{4}A\sqrt{s})$  and the H\"{o}lder inequality would imply 
\be\label{MVI 1}
\begin{split}
1&=|\Ric|(\bar x, 1)\\
&\leq c(n) |\Rm|(\bar x, 1)\\
&\leq C'(n,A)\left(\int^1_{1/2}\fint_{B_s(\bar x, \frac14A\sqrt{s})} |\Rm|^{n/2} d\mu_s ds \right)^{2/n}\\
&\leq C''(n,A) \e_0.
\end{split}
\ee
which is impossible if $\e_0\leq \e(n, A)$ is sufficiently small. Hence $\tilde T\geq T\wedge  \hat T(n, A)$. This implies the curvature estimate for $|\Rm|$ on $B_t(p,\frac14 A\sqrt{t})$ by Step 1. 

\bigskip 
{\bf Step 3. Improved curvature estimates.}
At this point we have already obtained a rough curvature estimate on $B_t(p,\frac12 A\sqrt{t}),t\in [0,T\wedge \hat T]$. For each $s\in [0,T\wedge \hat T]$, we may consider $\tilde g(t)=s^{-1}g(st),t\in [0,1]$. Since we have curvature bound on $[\frac12,1]$ and entropy lower bound, with the scaling invariant $L^{n/2}$ assumption we can apply iteration \cite{PLi-2012} again to show that 
\begin{equation}\begin{split}
&\quad|\Rm(\tilde g(x,1))|\\
&\leq C(n,A)\left(\fint^1_{1/2} \fint_{B_{\tilde g(s)}(x,\frac14 A\sqrt{s})} |\Rm(\tilde g(t))|^{n/2} d\mu_s ds \right)^{2/n}\\
&\leq C(n,A)\e_0.
\end{split}
\end{equation}
This gives an improved coefficient on curvature decay by rescaling it back to $g(t)$.
\end{proof}

We now show how our gap theorem for complete shrinking and steady gradient Ricci solitons with small $||\Rm||_{L^{n/2}}$, Theorem \ref{gap thm for soliton}, follows from Theorem \ref{local curvature}. Recall that a complete Riemannian manifold $(M, g)$ is said to be a shrinking (steady) gradient Ricci soliton if there exists a smooth function $f$ such that
\be\label{RS eqn}
\Ric+\nabla ^2 f=\lambda g,
\ee
where the constant $\lambda=\frac{1}{2}$ ($=0$ resp.).


\begin{proof}[Proof of Theorem \ref{gap thm for soliton}]
Let $\lambda =1/2$ or $0$ be the constant as in \eqref{RS eqn}. We consider the flow $\phi_t$ of the vector field $\frac{\nabla f}{1-2\lambda t}$ with $\phi_0$ being the identity map. it is known that $g(t):= (1-2\lambda t)\phi_t^*g$ is an  ancient solution to the Ricci flow on $M$ with $g(0)=g$ and $t$ $\in (-\infty, \frac{1}{2\lambda})$ ($=\mathbb{R}$ if $\lambda=0$, see \cite{{Chowetal-2007}, {Zhang-2009}}). By the reparametrization and the scaling invariance of Conditions $1$ and $2$ in Theorem \ref{gap thm for soliton}, we have for all $t$ $\in (-\infty, \frac{1}{2\lambda}):$
\begin{enumerate}
    \item $\nu(M,g(t))\geq -A$;
    \item $\int_M |\Rm|^{n/2}_{g(t)} d\mu_{g(t)} \leq \e$.
\end{enumerate}
We are going to show something slightly more general, namely if $(M, g(t))$ is a complete ancient solution to the Ricci flow on $(-\infty, 0]$ such that $g(t)$ satisfies the above two conditions for each $t$ $\in (-\infty, 0]$, then $(M, g(t))$ is isometric to $\mathbb{R}^n$. For any $Q>1$ and $\tau\leq 0$, we consider the rescaled solution $h(t):=(\frac{Q}{\hat T})^{-1}g(\frac{Q}{\hat T} t-Q+\tau)$, where $t$ $\in [0,\hat T ]$ and $\hat T$ is the constant as in Theorem \ref{local curvature}. It is not difficult to see that $h(t)$ also satisfies the two conditions in Theorem \ref{gap thm for soliton}. Hence we may apply Theorem \ref{local curvature} for all sufficiently small $\e$ to get for any $x$ $\in M$
\begin{eqnarray*}
Q |Rm|_g(x, \tau)&=& \hat T|Rm|_h(x, \hat T)\\
&\leq& C(n, A)\e.
\end{eqnarray*}
By letting $Q$ $\to \infty$, we have $g(\tau)$ is flat. The entropy lower bound at all scales then implies the maximal volume growth of $g(\tau)$ and thus it is isometric to $\mathbb{R}^n$.
\end{proof}

\section{Gap theorem with small $||\Rm||_{L^{n/2}}$}\label{gapriccisection}

In this section, we will use Ricci flow to discuss Riemannian manifolds with $\Ric\geq 0$ and with small $||\Rm||_{L^{n/2}}$ which are non-collapsed in term of entropy. We first show that under the assumption of Corollary~\ref{Corollary-diff}, we have a long-time solution of the Ricci flow and $g_0$ has maximal volume growth. 
\begin{thm}\label{Theorem-LT}
For any $A>0$, there is $\sigma(n,A),C_1(n,A)>0$ such that the following holds. Suppose $(M,g_0)$ is a complete non-compact Riemannian manifold with bounded curvature such that 
\begin{enumerate}
    \item $\Ric(g_0)\geq 0$;
    \item $\nu(M,g_0)\geq -A$;
    \item $\left(\int_M |\Rm(g_0)|^{n/2} d\mu_{g_0}\right)^{2/n} \leq \e$ for some $\e<\sigma$.
\end{enumerate}
Then there is a Ricci flow $g(t)$ starting from $g_0$ on $M\times [0,\infty)$ such that for all $t>0$, 
\begin{equation}
    \left\{
    \begin{array}{ll}
         & \sup_M |\Rm(g(t))|\leq C_1\e t^{-1}\\
         & \left(\int_M |\Rm(g(t))|^{n/2} d\mu_t\right)^{2/n} \leq C_1\e
    \end{array}
    \right.
\end{equation}
Moreover, $g_0$ is of maximal volume growth.
\end{thm}
{
\begin{rem}
 The assumption on the global entropy of all scale can also be implied by  maximal volume growth. 
\end{rem}}
\begin{proof}
For $R>0$, we let $g_{R,0}=R^{-2}g_0$ which still satisfies the assumptions of the Theorem, which are scaling invariant. Therefore we can run Shi's Ricci flow $g_{R}(t)$ \cite{Shi1989} for a short-time with initial metric $g_{R,0}$. By Theorem~\ref{pseudo}, if $\sigma$ is sufficiently small, $g_R(t)$ exists on $M\times [0,T(n,A)]$ and satisfies 
\begin{equation}
    \left\{
    \begin{array}{ll}
         &  |\Rm(g_R(t))|\leq C_1\e t^{-1}\\
         & \left(\int_{B_{g_R(t)}(x,1)}|\Rm(g_R(t))|^{n/2} d\mu_{R,t}\right)^{2/n} \leq C_1 \e 
    \end{array}
    \right.
\end{equation}
for all $(x,t)\in M\times [0,T]$. By re-scaling it back and the uniqueness of Ricci flow \cite{ChenZhu2006}, we obtain a Ricci flow $g(t)$ on $[0,T R^2)$ with $|\Rm|\leq C_1\e t^{-1}$ and $g(0)=g_0$. Moreover, we have for all $R,t>0$, 
\begin{equation}
    \left( \int_{B_t(x,R)} |\Rm(g(t))|^{n/2}d\mu_t\right)^{2/n}\leq C_1 \e.
\end{equation}
The global integral estimate then follows by letting $R\rightarrow +\infty$. 

To see that $g_0$ is of maximal volume growth, thanks to the improved regularity on curvature and monotonicity of entropy $\nu$, the re-scaled Ricci flow $g_R(t)$ satisfies 
\begin{eqnarray}
\mathrm{Vol}_{g_R(1)}(B_{g_{R}(1)}(x,1)) \geq c.
\end{eqnarray} 
Since the lower bound of scalar curvature is preserved along the Ricci flow, together with \cite[Corollary 3.3]{SimonTopping2016}, we have, if $\sigma$ is sufficiently small, that
\begin{equation}
\begin{split}
c&\leq \mathrm{Vol}_{g_{R}(1)}(B_{g_{R}(1)}(x,1))     \\
&\leq \mathrm{Vol}_{g_{R}(0)}(B_{g_{R}(0)}(x,2)) \\
&= \frac{\mathrm{Vol}_{g_0}(B_{g_0}(x,2R))}{R^n}.
\end{split}
\end{equation}
Since $R$ is arbitrarily large, this completes the proof.

\end{proof}

{ 
Before we prove the Corollary~\ref{Corollary-diff}, we will show that the asymptotic volume ratio can be improved to be almost Euclidean if we further shrink the integral curvature and hence is almost Euclidean in the sense of local entropy \cite[Lemma 4.10]{Wang-2020}. This is in spirit similar to the gap theorem proved by Cheeger \cite[Theorem 4.32]{Cheeger2003}. 
\begin{thm}\label{gap-entrop}
For all $A,\lambda,\delta>0$, there are $\sigma(n,A,\lambda,\delta), r(n,A,\lambda,\delta)>0$ such that if $(M,g)$ is a complete Riemannian manifold of bounded curvature so that for all $p\in M$, 
\begin{enumerate}
    \item $\Ric(g)\geq -\lambda$;
    \item $\nu\left(B_{g}(p,5),g,1 \right)\geq -A$;
    \item $\left(\int_{B_g(p,2)}|\Rm(g)|^{n/2} d\mu_{g}\right)^{2/n} <\sigma$.
\end{enumerate}
Then for all $p\in M$,
\begin{eqnarray}
\mathrm{Vol}_g\left(B_g(p, r)\right)\geq (1-\delta) \omega_n  r^n.
\end{eqnarray}
\end{thm}
\begin{proof}
Let $g(t),t\in [0,\hat T]$ be the Ricci flow solution obtained from Theorem~\ref{pseudo} and Shi's Ricci flow \cite{Shi1989}. We claim that for given $\delta$, there are constants $T(n,A,\lambda,\delta)$ and $\hat \sigma(n,A,\lambda,\delta)$ such that if $\sigma<\hat \sigma$, then for $(x,t)\in M\times [0,T]$,
\begin{equation}\label{flow volume lower bdd}
    \mathrm{Vol}_{g(t)}\left(B_{g(t)}(x,\sqrt{t})\right) \geq (1-\delta)\omega_n t^{n/2}.
\end{equation}

Suppose on the contrary, we can find a sequence of $g_i(t),t\in [0,\hat T]$ 
such that $g_i(0)$ satisfies the same assumptions as in Theorem~\ref{pseudo} with $\e_i\rightarrow 0$ now but for some $\hat  x_i\in M_i$ and $0<\sqrt{t_i}\rightarrow 0$,
\begin{equation}\label{Con-pt}
   \mathrm{Vol}_{g_i(t_i)}\left(B_{g_i(t_i)}(\hat x_i,\sqrt{t_i})\right) < (1-\delta) \omega_n t_i^{n/2}. 
\end{equation}
We note here that since $\e_i\to 0$, the existence time $\hat{T}$ can be chosen to be independent of $i\to+\infty$ by Theorem~\ref{pseudo}.

Consider the rescaled Ricci flow $\tilde g_i(t)=t_i^{-1}g(t_i t)$ on $M_i\times [0,1]$. The original estimates from Theorem~\ref{pseudo} imply that for all sufficiently large $i$ and all $(x,t)\in M_i\times (0,1]$, 
\begin{eqnarray*}
|\Rm(\tilde g_i(t))|\leq C_0(n,A,\lambda)\e_i t^{-1}\quad\text{and}\quad \mathrm{inj}_{\tilde g_i(t)}\geq c_0(n,A,\lambda)\sqrt{t}
\end{eqnarray*}
which enable us to pass $(M_i,\tilde g_i(t),\hat x_i)$ to a  sub-sequential limit $(M_\infty,\tilde g_\infty(t),\hat x_\infty)$ in smooth Cheeger-Gromov sense by Hamilton's compactness \cite{Hamilton-1995.2}. In particular, $\tilde g_\infty(t)$ is flat for $t\in (0,1]$ since $\e_i\to 0$. On the other hand, since $t_i\rightarrow 0$, we may apply the local monotonicity of entropy in \cite[Theorem 5.4]{Wang-2018} again to show that $\nu(M_\infty,\tilde g_\infty(t))\geq -2A$ which implies $\tilde g_\infty(t)$ is of Euclidean volume growth by Theorem~\cite[Theorem 3.3]{Wang-2018} and hence $(M_\infty,\tilde g_\infty(1))\equiv (\mathbb{R}^n,g_{euc})$ which contradicts \eqref{Con-pt}.

 After relabelling the constants, \eqref{flow volume lower bdd}, together with volume comparison implies that for all $t\in [0,T]$ and $r<\sqrt{t}$, 
\begin{eqnarray}
\mathrm{Vol}_{g(t)}\left(B_{g(t)}(x,r)\right) \geq \left(1-\frac\delta{2}\right)\omega_n r^n.
\end{eqnarray}

 By the scalar curvature lower bound of $g(t)$ and \cite[Corollary 3.3]{SimonTopping2016},
\begin{equation}
    \begin{split}
    \mathrm{Vol}_{g_0}\left(B_{g_0}(x,r) \right) 
    &\geq e^{-n\lambda t}\cdot \mathrm{Vol}_{g(t)}\left(B_{g_0}(x,r )\right) \\
    &\geq e^{-n\lambda t}\cdot \mathrm{Vol}_{g(t)}\left(B_{g(t)}(x,r-c_n\sqrt{\sigma t}) \right) \\
    &\geq e^{-n\lambda t} \left(1-\frac\delta{2}\right)\omega_n \left( r-c_n\sqrt{\sigma t}\right)^n\\
    &\geq  \left(1-\delta\right)\omega_n r^n
    \end{split}
\end{equation}
if we choose $t,\sigma$ small enough and $r=\frac{1}2 \sqrt{t}$. This completes the proof.
\end{proof}

}

\begin{proof}[Proof of Corollary~\ref{Corollary-diff}]
 
Theorem~\ref{Theorem-LT} implies that $g_0$ is of maximal volume growth. By \cite[Theorem A.1.11]{CheegerColding1997} (see also \cite[Theorem 5.7]{Wang-2020}), it suffices to show that the asymptotic volume growth can be made arbitrarily close to the Euclidean one if we shrink $\sigma$. This follows from Theorem~\ref{gap-entrop} and the rescaling argument as in the proof of Theorem~\ref{Theorem-LT}.



Alternatively, we can also prove the homeomorphism by showing $M=\bigcup_{i=1}^\infty U_i$ where $U_i$ is diffeomorphic to a Euclidean ball and $U_i\subset U_{i+1}$ for all $i$ using the expansion of injectivity radius, curvature estimate $|\Rm(x,t)|\leq \e t^{-1}$ from Theorem~\ref{Theorem-LT}. Then the homeomorphism will follow from the main result of \cite{Brown1961}, see also \cite[Section 3]{ChenZhu-2003}. Notice that Gompf's result says that among
the Euclidean spaces only $\mathbb{R}^4$ has exotic differential structures. So for $n>4$, the homeomorphisms can be made to be diffeomorphisms (see \cite{Stallings-1962}). 
\end{proof}

\section{Regularity of Gromov-Hausdorff limit}\label{ghsection}
In this section, we discuss the compactness of Riemannian manifolds satisfying small $L^{n/2}$ bound. We remark here that the Gromov-Hausdorff limit follows from Ricci lower bound directly. The key part is to construct the differentiable structure on the limit using the pseudolocality of Ricci flows.

\begin{proof}
[Proof of Theorem~\ref{Theorem:GH}]
By Shi's Ricci flow existence \cite{Shi1989} and Theorem~\ref{pseudo}, by choosing $\e_0$ small enough we can find a sequence of Ricci flow $g_i(t)$ on $M_i\times [0,T(n,A)]$ such that 
\begin{enumerate}
    \item $\Ric(g_i(0))\geq -\lambda$;
    \item $\nu(B_{g_i(t)}(x,1),g_i(t),\frac1{32})\geq -2A$;
    \item $|\Rm(g_i(t))|\leq \frac{C\e_0}{t}$
\end{enumerate}
for all $(x,t)\in M_i\times (0,T]$. By \cite[Theorem 3.3]{Wang-2018} and \cite{CheegerGromovTaylor-1982}, we can apply Hamilton's compactness to pass  $(M_i,g_i(t),p_i)$ to $(M_\infty,g_\infty(t),p_\infty)$ for $t\in (0,T]$ in the smooth Cheeger-Gromov sense after passing to  sub-sequence. More precisely, there is an exhaustion $\{\Omega_i\}_{i=1}^\infty$ of $M_\infty$ and a sequence of diffeomorphism $F_i:\Omega_i\rightarrow M_i$ onto its image such that for any compact subset $\Omega\times [a,b]\Subset M_\infty\times (0,T]$, we have  $F_i^*g_i(t)\rightarrow g_\infty(t)$ in $C^\infty_{loc}(\Omega\times [a,b])$.
  
We now construct the Gromov-Hausdorff limit of $g_i$ using $F_i$ in a more precise way so that its relation to $M_\infty$'s topology is clearer. This essentially follows from the proofs of Gromov's compactness theorem and the distance distortion estimates. Since $M_\infty$ is a smooth manifold, we let $\{x_k\}_{k=1}^\infty$ be a countable dense set with respect to $g_\infty(1)$. Then for each $k,l$, we have $x_k,x_l\in B_{g_\infty(1)}(p_\infty,R_{k,l})$ and hence by distance distortion estimates \cite[Corollary 3.3]{SimonTopping2016} using curvature estimates above, we have 
\begin{equation}
    d_{F_i^*g_i}(x_k,x_l)\leq d_{F_i^*g_i(1)}(x_k,x_l)+ C_n\leq C(k,l)
\end{equation}
as $i\rightarrow +\infty$.
  Here we have used the fact that $F_i^*g_i(1)$ converges locally uniformly to $g_\infty(1)$. Therefore, $\lim_{i\rightarrow +\infty}d_{F_i^*g_i}(x_k,x_l)$ exists after we pass it to some sub-sequence which we denote it as $d_\infty(x_k,x_l)$. Repeating the process for each $k,l$, we define $d_\infty$ on the dense set. For general $x,y\in M_\infty$, we define $d_\infty(x,y)$ using the density of $\{x_k\}$. This is well defined since if there are two sequences $x_i,x_i'\rightarrow x\in M_\infty$ and $y_i,y_i'\rightarrow y\in M_\infty$ with respect to $g_\infty(1)$, then for $i$ sufficiently large,
  \begin{equation}
      \begin{split}
          d_\infty(x_i,y_i)&\leq   d_\infty(x_i',y_i')+d_\infty(x_i,x_i')+d_\infty(y_i,y_i')\\
          &\leq d_\infty(x_i',y_i')+ C\left(d_{g_\infty(1)}(x_i,x_i')\right)^{1/2}+C\left(d_{g_\infty(1)}(y_i,y_i')\right)^{1/2}\\
          &=d_\infty(x_i',y_i') +o(1).
      \end{split}
  \end{equation}
  by using \cite[Lemma 2.4]{{HuangKongRongXu2018}} and \cite[Corollary 3.3]{SimonTopping2016}. By passing $i\rightarrow +\infty$ and switching the sequences, we have the uniqueness of the limit. In other words, we have 
 \begin{equation}\label{GH-dis-conver}
     \lim_{i\rightarrow +\infty}d_{F_i^*g_i}(x,y)=d_\infty(x,y)
 \end{equation} 
for all $x,y\in M_\infty$.
  
  Now we claim that $d_\infty(\cdot,\cdot)$ is in fact a distance defined on $M_\infty\times M_\infty$. To see this, let $y,z\in M_\infty$ be such that $d_\infty(z,y)=0$. If $y\neq z$, then we have $d_{g_\infty(1)}(z,y)>r$ for some $r>0$. For any $\e>0$, we can find $y',z'\in \{x_i\}_{i=1}^\infty$ such that $d_{g_\infty(1)}(y,y')+d_{g_\infty(1)}(z,z')+ d_{\infty}(y',z')<\e$ and therefore we can find $N\in \mathbb{N}$ such that for $i>N$, $d_{F_i^*g_i}(y',z')<3\e$. Applying \cite[Lemma 2.4]{{HuangKongRongXu2018}} again, we deduce 
  \begin{equation}
      d_{F_i^*g_i(1)}(y',z')\leq C(n,\lambda) \e^{2/3}. 
  \end{equation}
  Here we note that although \cite[Lemma 2.4]{{HuangKongRongXu2018}} is stated globally, it is easy to see that the proof holds locally and only require the curvature bound in form of $\e t^{-1}$ for $\e$ small enough and an initial Ricci lower bound which is avaliable in our situation. Therefore, if $\e$ is sufficiently small, it will violate the fact that $d_{g_\infty(1)}(y,z)>r$. This shows that $d_\infty$ defines a distance metric on $M_\infty$.  
  
  To see that $d_\infty$ generates the same topology as $M_\infty$, it suffices to point out that \cite[Lemma 2.4]{{HuangKongRongXu2018}} together with a limiting argument implies that for $d_\infty(x,y)<1$,  we have 
  \begin{eqnarray}\label{GH ball}
  C_n^{-1}d_{g_\infty(1)}(x,y)^{3/2}\leq d_\infty(x,y)\leq C_nd_{g_\infty(1)}(x,y)^{1/2}
  \end{eqnarray}
  and hence all small open balls are uniformly comparable. 
  Moreover by \cite[Lemma 2.2]{LeeTam2020}, we also see that $\{B_{d_\infty}(p_\infty,k)\}_{k=1}^\infty$ is an exhaustion of $M_\infty$. By the construction, \eqref{GH-dis-conver}, and \eqref{GH ball}, the pointed Gromov-Hausdorff convergence is straight forward with $F_i$ being the Gromov-Hausdorff approximation on each compact set $\Omega\Subset M_\infty$. 
\end{proof}

\end{document}